\documentclass[11pt]{amsart}

\usepackage{latexsym,amssymb,amscd,amsmath,mathrsfs}

\usepackage[all]{xy}

 \newtheorem{theorem}{Theorem}[section]
 \newtheorem{prop}[theorem]{Proposition}
 \newtheorem{corol}[theorem]{Corollary}
 \newtheorem{lemma}[theorem]{Lemma}
 \newtheorem*{main}{Theorem}

\numberwithin{equation}{section}

 \def\iff{if and only if }

 \def\sG{\mathsf G}	\def\sH{\mathsf H}
 \def\sS{\mathsf S}	\def\sN{\mathsf N}
 \def\mC{\mathbb C}	\def\mX{\mathbb X}
 \def\mZ{\mathbb Z}	\def\mN{\mathbb N}
 \def\mK{\mathbb K}	\def\mP{\mathbb P}
 \def\mA{\mathbb A} \def\Mk{\Bbbk}
 \def\bA{\mathbf A}	\def\bB{\mathbf B}
 \def\bW{\mathbf W}	\def\bH{\mathbf H}	
 \def\fD{\mathbf d}	\def\fP{\mathbf P}
 \def\fH{\mathbf h}	\def\fX{\mathbf x}
 \def\rA{\mathrm A}	\def\rD{\mathrm D}
 \def\rE{\mathrm E}	\def\rH{\mathrm H}
 \def\rQ{\mathrm Q}
 \def\cC{\mathscr C}	\def\cD{\mathscr D}
 \def\cH{\mathscr H}	\def\cP{\mathscr P}
 \def\cI{\mathscr I}	\def\cR{\mathscr R}
 \def\cT{\mathscr T}
 \def\kC{\mathcal C}	\def\kP{\mathcal P}
 \def\kQ{\mathcal Q}	\def\kH{\mathcal H}
 \def\kT{\mathcal T}	\def\kO{\mathcal O}
 \def\kF{\mathcal F}	\def\kD{\mathcal D}
 \def\gM{\mathfrak m}
 \def\al{\alpha}		\def\be{\beta}
 \def\si{\sigma}		\def\De{\Delta}
 \def\la{\lambda}	\def\th{\theta}
 \def\La{\Lambda}

 \def\setsuch#1#2{\left\{\,#1\mid #2\,\right\}}
\def\gnr#1{\langle\,#1\,\rangle}
 \def\row#1#2{\left( #1_1 , #1_2 , \dots , #1_{#2} \right)}
 \def\lst#1#2{ #1_1 , #1_2 , \dots , #1_{#2} }
 \def\sbe{\subseteq}	\def\mps{\mapsto}
 \def\mtr#1{\begin{pmatrix}#1\end{pmatrix}}
 \def\+{\oplus}		\def\*{\otimes}
 \def\bop{\bigoplus}	\def\xx{\times}
 \def\ti{\tilde}		\def\xarr{\xrightarrow}
 \def\8{\infty}		\def\bup{\bigcup}
 \def\md{\mbox{-}\mathrm{mod}}
  	\def\op{^\mathrm{op}}
 \def\df{\mbox{-}}		\def\ld{{^\wedge}}
 \def\Arr{\Rightarrow}	\def\larr{\longrightarrow}
 \def\ol{\overline}	\def\dd{\partial}
 \def\={\setminus}
 \def\End{\mathop\mathrm{End}\nolimits}
 \def\Aut{\mathop\mathrm{Aut}\nolimits}
 \def\Hom{\mathop\mathrm{Hom}\nolimits}
 \def\Ker{\mathop\mathrm{Ker}}
 \def\im{\mathop\mathrm{Im}}
 \def\Mat{\mathop\mathrm{Mat}}
 \def\cok{\mathop\mathrm{Cok}}
 \def\Md{\mbox{-}\mathrm{Mod}}
 \def\pmd{\mbox{-}\mathrm{mod^p}}
 \def\Pmd{\mbox{-}\mathrm{Mod^p}}
 \def\Pro{\mbox{-}\mathrm{Proj}}
 \def\pro{\mbox{-}\mathrm{proj}}
 \def\per{^\mathrm{per}}
 \def\gd{\mathop\mathrm{gl.dim}}
 \def\gl{\mathop\mathrm{GL}}
 \def\El{\!\mbox{-}\mathrm{El}}
 \def\el{\!\mbox{-}\mathrm{el}}
 \def\Dim{\mathop\mathbf{dim}\nolimits}
 \def\pd{\mathop\mathrm{pr.dim}\nolimits}
 \def\ext{\mathop\mathrm{Ext}\nolimits}
 \def\rep{\mathop\mathrm{rep}\nolimits}
 \def\rpp{\mathrm{rep}^{\mathrm p}}
 \def\tl{\otimes^{\mathbb L}}		\def\pe{\preccurlyeq}
 \def\rh{\mathop{\mathbb R\mathrm{Hom}}\nolimits}
 \def\hN{\hat\sN}	\def\hP{\hat P} 	\def\lt{\ltimes}
 \def\bd{\boldsymbol\delta}	\def\1{\mathbf1}
 	\def\ind{\mbox{-}\mathrm{ind}}
 \def\vd{vector dimension}	\def\cod{\mathop\mathrm{codim}\nolimits}
 \def\stab{\mathop\mathrm{Stab}}
 \def\hG{\hat\sG}	\def\hH{\hat\sH}	\def\tG{\tilde\sG}
 \def\hgl{\mathop{\widehat{\mathrm{GL}}}\nolimits}

\begin{document}

 \title[Tilting, deformations and representations of linear groups]
{Tilting, deformations and representations\\ of linear groups over
 Euclidean algebras}
 \author{Viktor Bekkert}
 \author{Yuriy Drozd}
 \author{Vyacheslav Futorny}
 \address{Departamento de Matem\'atica, Instituto de Ci\^encias Exatas,
 Universidade Federal de Minas Gerais, Belo Horizonte, Brazil}
 \email{bekkert@mat.ufmg.br} 
 \address{Institute of Mathematics, National Academy of Sciences of Ukraine,
 Kiev, Ukraine}
 \email{drozd@imath.kiev.ua}
 \address{Instituto de Matem\'atica e Estat\'istica, Universidade de S\~ao Paulo,
 S\~ao Paulo, Brazil}
 \email{futorny@ime.usp.br}
 \begin{abstract}
  We consider the dual space of linear groups over Dynkinian and Euclidean algebras, i.e.
 finite dimensional algebras derived equivalent to the path algebra of Dynkin or Euclidean quiver.
 We prove that this space contains an open dense subset isomorphic to the product of
 dual spaces of full linear groups and, perhaps, one more (explicitly described) space.
 The proof uses the technique of bimodule categories, deformations and representations
 of quivers.
 \end{abstract}
 \maketitle

 \section*{Introduction}
 \label{intro}

 Let $\sG$ be a Lie group. Denote by $\hG$ its \emph{dual space}, i.e. the space of
 irreducible unitary representations of $G$ in Hilbert spaces. It is a topological (though
 often non-Hausdorff) space \cite{kir}; moreover, if $\sG$ is of \emph{type~I} (or
 \emph{domestic}), there is a natural measure $\mu_\sG$ (the \emph{Plancherel measure})
 on $\hG$, which inables a harmonic analysis on $\sG$ \cite{kl}. The structure of $\hG$
 is thoroughly investigated in two cases: that of \emph{solvable} groups, when the
 \emph{orbit method} \cite{kir} gives rather complete information, and that of
 \emph{reductive} groups. Essentially less is known about the ``\emph{mixed}'' groups,
 i.e. those which are neither solvable nor reductive. For instance, in the classical survey
 \cite{zs} only two rather simple examples of such groups are considered. 

 Certain explanation of this situation was given by Li Sun Gen \cite{li}. Namely, he showed
 that if we deal with the groups of $2\xx2$ block triangular matrices 
 \[
  \sG(m,n)=\mtr{A&B\\0&C},\ A\in\gl(m),\ C\in\gl(n),\ B\in\Mat(m\xx n),
 \]
 then the problem of classification of unitary representations of these groups (for all values of
 $m,n$) is ``\emph{wild},'' i.e. contains the problem of classification of all finite dimensional
 representations of all finitely generated algebras, which seems hopeless. Nevertheless, the same
 paper proposed a certain approach to this problem. Namely, it was shown there that the space
 $\hG(m,n)$ contains an open dense subset $\tG$ isomorphic to $\hgl(d)$,
 where $d=\gcd(m,n)$. Moreover, if $\sG$ is any group of block triangular matrices (with
 any number of blocks), then $\hG$ contains an open dense subset $\tG$ isomorphic to
 $\prod_{i=1}^s\hgl(d_i)$ for some $s$ and $d_i$.

 These results were generalized by A.~Timoshin and the author \cite{ti0,dr}. The final result
 of \cite{dr} dealt with the \emph{linear groups over Dynkinian algebras}. We call
 \emph{linear groups} over an algebra $\bA$ the groups
 $\gl(P,\bA)=\Aut_\bA P$, where $P$ is a finitely generated projective $\bA$-module.
 If $P$ is free of rank $r$, then $\gl(P,\bA)\simeq\gl(r,\bA)$, but in general case
 there are more possibilities. For instance, if $\bA$ is the algebra of triangular matrices,
 we obtain as $\gl(P,\bA)$ the groups of block triangular matrices. A finite dimensional algebra
 $\bA$ is called \emph{Dynkinian} if it is \emph{derived equivalent} to the path algebra of a
 \emph{Dynkin quiver}, i.e. oriented graph such that its underlying non-oriented graph is a
 Dynkin scheme $\rA_n,\rD_n$ or $\rE_{6,7,8}$.  It was proved in \cite{dr} that if
 $\sG=\gl(P,\bA)$, where $\bA$ is a Dynkinain algebra over a locally compact field $\mK$,
 then $\hG$ contains an open dense subset $\tG\simeq\prod_{i=1}^s\gl(d_i,\mK)$.
 Using the results of \cite{kl}, one can show that this isomorphism reflects the Plancherel
 measure and $\tG$ is a support of this measure, so, in some sence, the obtained information
 is sufficient for the harmonic analysis. Note that all groups $\gl(P,\bA)$ are linear algebraic
 groups, so of type~I, thus they fit the framework of \cite{kl}.
 The proof used rather delicate matrix calculations in terms of the so
 called \emph{representations of boxes} and \emph{small reduction algorithm}.

 For linear groups over other algebras the picture becomes more complicated. In \cite{ti1,ti2}
 Timoshin considered the cases when $\bA$ is the \emph{Kronecker algebra}, i.e. the path
 algebra of the quiver
 \[
 \bullet \rightrightarrows \bullet,
 \] 
 and the path algebra of the quiver of type $\ti\rA_2$, i.e.
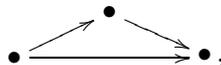
 \[
 \xymatrix@R=1ex{& \bullet \ar[dr] \\ \bullet \ar[ur] \ar[rr] && \bullet\,, }
 \]
 as well as over derived equivalent algebras over the field $\mC$ of complex numbers.
 He proved that in this case the dual space to a linear group $\gl(P,\bA)$ contains an open
 dense subset $\tG$ isomorphic to
 \[
 \prod_{i=1}^s\gl(d_i,\mC)\xx \mC^{(m)}/\sS_m\xx (\mC^\xx)^m,
 \]
 where $\mC^{(m)}=\setsuch{\row\la m}{\la_i\in\mC,\,\la_i\ne\la_j \text{  for } i\ne j}$,
 $\sS_m$ is the symmetric group naturally acting on $\mC^{(m)}$, $\mC^{(m)}/\sS_m$
 is the corresponding factor-space and $\mC^\xx$ is the
 multiplicative group of the field $\mC$. The proof was based again on matrix
 calculations, in particular, on the small reduction algorithm (slightly generalized).

 The aim of our paper is to generalize the last result to the linear groups over
 \emph{Euclidean algebras}, i.e. those, which are derived equivalent to the path algebras
 of Euclidean (extended Dynkin) quivers of type $\ti\rA_n,\,\ti\rD_n$ or $\ti\rE_{6,7,8}$.
 Namely, we prove the following result (Theorem~\ref{t52}):

  \begin{main}
     Let $\bA$ be a Euclidean algebra, $\sG=\gl(P,\bA)$ be a linear group over $\bA$.
 There is a subset $\tG\sbe\hG$, which is open, dense and support of the Plancherel measure,
  such that 
  \[
 \tG\simeq\prod_{i=1}^s\hgl(d_i,\mC) \xx \mX^{(m)}/\sS_m \xx (\mC^\xx)^m 
 \]
 for some values of $\lst ds$, $m$ and some cofinite subset $\mX\sbe\mP^1$.
 \end{main}
 
 The main distinction is that we use, instead of matrix calculations, the deformation technique,
 the \emph{tilting theory} and the results on the representations of Euclidean quivers
 from \cite{kac,rib}. Since it simplifies the considerations, we also included a new
 simple proof of the result for Dynkinian algebras from \cite{dr}. Actually, we had to
 extend the deformation and tilting theories to the \emph{bimodule categories}, in
 particular, to prove for them analogues of the well-known
 Riedtmann--Zwara theorem about deformations of modules \cite{rie,zw}. It is done
 in Sections~1 and~2. In section~3 we recall the facts concerning the Tits and the Euler
 forms and extend them to bimodule categories too. Section~4 contains the basic
 results on elements in general position in the bimodule categories of Dynkinian and
 Euclidean bimodules, and in Section~5 these results are applied to the proof of the
 main theorems about representations of linear groups over Dynkinian and
 Euclidean algebras. 

 Note that, though our proofs are essencially simpler than those from \cite{dr}, they are
 not so constructive. To get explicit shape of representations from the subset $\tG$, one
 still has to apply a sort of the small reduction algorithm. Our results then guarantee that
 one gets an answer after finitely many steps. Typical examples can de found in \cite{dr,ti1,ti2}.

 \smallskip\emph{Acknowledgements.} This investigation was accomplished during the visit
 of the second author to the University of S\~ao Paulo supported by Fapesp (processo 2007/05047-4).
 The second author was also partially supported by INTAS Grant 06-1000017-9093.
 The third author was partially supported by Fapesp (processo 2005/60337-2) and CNPq (processo
 301743/2007-0).
  
 \section{Derived bimodule categtories}
 \label{s1}
 
 Let $\bA$ be a ring. We denote by $\bA\Md$ the category of all (left) $\bA$-modules and 
 by $\bA\md$ the category of finitely generated (left) $\bA$-modules. By $\bA\Pro$ we denote
 the subcategory of projective $\bA$-modules and by $\bA\pro$ the subcategory of finitely generated
 projective $\bA$-modules.  Let $\bW$ be an $\bA$-bimodule. The \emph{bimodule category}
 (or the \emph{category of elements of the bimodule}, or of \emph{matrices over the bimodule}
 $\bW$) $\bW\El$ is defined as follows \cite{db}.
 \begin{itemize}
 \item  The \emph{objects} of $\bW\El$ are elements of groups
 $\bW(P)=\Hom_\bA(P,\bW\*_\bA P)$, where $P\in\bA\Pro$; we call them \emph{elements
 of the bimodule} $\bW$.
 \item  A \emph{morphism} from $w\in\bW(P)$ to $w'\in\bW(P')$ is a homomorphism
 $\al:P\to P'$ such that $(1\*\al) w=w'\al$. We denote the set of such morphisms by $\Hom_\bW(w,w')$.
 \item  The product of morphisms coincides with their compositions as maps.
 \end{itemize}
 The category $\bW\El$ is additive; moreover, it is \emph{fully additive} (Karoubian), i.e. every its
 idempotent endomorphism in it splits. We denote by $\bW\el$ the full subcategory consisting of
 elements of the groups $\bW(P)$ with $P\in\bA\pro$.

 The bimodule category is an \emph{exact category} in the sense of \cite{qu} (we use the terminology
 of \cite{gr} or \cite[Appendix~A]{kel}). Namely, \emph{conflations} (distinguished exact sequences)
 in $\bW\El$ are sequences $w_1\xarr\al w_2\xarr\be w_3$, where $w_i\in\bW(P_i)$, such
 that the underlying sequence $0\to P_1\xarr\al P_2\xarr\be P_3\to0$ is exact. Thus,
 $\al:w_1\to w_2$ is an \emph{inflation} if the underlying map $\al:P_1\to P_2$ is a split
 monomorphism, and $\be:w_2\to w_3$ is a \emph{deflation} if the underlying map
 $\be:P_2\to P_3$ is an epimorphism (automatically split, since $P_3$ is projective). Therefore,
 one can consider its \emph{derived category} $\cD(\bW\El)$ \cite{nee}. We denote
 by $\cC(\bW\El)$ the category of complexes over $\bW\El$ and by $\cH(\bW\El)$ its factor-category
 modulo homotopy. A complex $\kC=(\kC_n,d_n)$ is said to be \emph{exact} if the underlying complex
 of projective $\bA$-modules is exact and $\Ker d_n$ splits for every $n$ (then also $\im d_n$
 splits for every $n$). Note that if this complex is \emph{right bounded}, i.e. $\kC_n=0$ for $n\ll0$,
 it is exact \iff the underlying complex of modules is exact. The derived category $\cD(\bW\El)$ is the
 localization of $\cH(\bW\El)$ with respect to the full subcategory of exact complexes (it is a triangular
 and \'epaisse subcategory \cite{nee}). As usually, we denote by marks $^-$ and $^b$ the full subcategories
 of $\cC(\bW\El),\,\cH(\bW\El)$ and $\cD(\bW\El)$ consisting, respectively, of right bounded and
 two-sided bounded complexes. The full subcategories of $\cC(\bW\El),\,\cH(\bW\El)$ and $\cD(\bW\El)$
 consisting of complexes over $\bW\el$ are denoted by $\cC(\bW\el),\,\cH(\bW\el)$ and
 $\cD(\bW\el)$ (with the marks $^-$ or $^b$, if necessary). The set of morphisms $\kC\to \kC'$
 in $\cD(\bW\El)$ will be denoted by $\Hom_{\cD\bW}(\kC,\kC')$.

 The most important case is that of \emph{bipartite bimodules}. Recall that an $\bA$-bimodule
 $\bW$ is said to be \emph{bipartite} if $\bA=\bA_1\xx\bA_2$ and $\bW$ is an
 $\bA_2\df\bA_1$-bimodule, where the $\bA$-action is defined as $(a_1,a_2)w=a_2w$ and
 $w(a_1,a_2)=wa_1$. If $P$ is a projective $\bA$-module, then $P=P_1\+P_2$, where $P_i$
 is a projective $\bA_i$-bimodule, and
 $\bW(P)=\bW(P_2,P_1)=\Hom_{\bA_2}(P_2,\bW\*_{\bA_1}P_1)$. A morphism $\al:w\to w'$,
 where $w\in\bW(P_2,P_1),\,w'\in\bW(P'_2,P'_1)$ is then a pair $(\al_1,\al_2)$, where
 $\al_i:P_i\to P'_i$, such that $(1\*\al_1)w=w'\al_2$.

 In what follows we mainly deal with the situation, when an $\bA_2\df\bA_1$-bimodule arises
 as the \emph{left dual} $\,\ld\bW=\Hom_{\bA_1}(\bW,\bA_1)$ of a $\bA_1\df\bA_2$-bimodule
 $\bW$. If $\bW$ is finitely generated as $\bA_1$-module, the category $\ld\bW\El$ can be interpreted as
 that of some modules over the ring $\bA=\bA_1[\bW]\bA_2$ of triangular matrices of the form
 $\mtr{a_1&w\\0&a_2}$, where $a_i\in\bA_i,\ w\in\bW$. Namely, there is an isomorphism
 of functors (on $\bA\Pro$):
 \[
 \ld\bW\*_{\bA_1}P_1 \simeq \Hom_{\bA_1}(\bW,P_1)
 \]
 (mapping $f\*p\,$ to the homomorphism $w\mps f(w)p$, where $p\in P_1,\,f\in\ld\bW,\,
 w\in\bW$). Therefore,
 \[
 \ld\bW(P_2,P_1)\simeq\Hom_{\bA_2}(P_2,\Hom_{\bA_1}(\ld\bW,P_1)) \simeq
 \Hom_{\bA_1}(\ld\bW\*_{\bA_2}P_2,P_1). 
 \]
  It implies

 \begin{prop}\label{p11}
  Let $\bA=\bA_1[\bW]\bA_2$, where $\bW$ is finitely generated as left $\bA_1$-module.
 Then $\ld\bW\El$ is equivalent to the full subcategory $\bA\Pmd$
 of  the category $\bA\Md$ consisting of all modules $M$ such that $e_iM\ (i=1,2)$ is a 
 projective $\bA_i$-module, where $e_1=\mtr{1&0\\0&0},\ e_2=\mtr{0&0\\0&1}$.
 \end{prop}
 \begin{proof}
  Let $M\in\bA\Pmd$, $P_i=e_iM$. The multiplication with $e_1\bA e_2\simeq\bW$ induces a homomorphism
 $\phi:\bW\*_{\bA_2}P_2\to P_1$ and the module $M$ is uniquely defined by the triple $(P_1,\phi,P_2)$ (we
 usually write $M=(P_1,\phi,P_2)$), hence, by an element $f(M)\in\ld\bW(P_2,P_1)$. Every homomorphism
 $\al:M\to M'$, where $M'=(P_1',\phi',P_2')$, maps $P_i$ to $P_i'$, hence, defines a pair $(\al_1,\al_2),\ 
 \al_i:P_i\to P'_i$, such that $\al_1\phi=\phi'(1\*\al_2)$. One easily sees that the last condition
 is equivalent to the equality $(1\*\al_1)f(M)=f(M')\al_2$, thus
 $(\al_1,\al_2)\in\Hom_{\ld\bW}(f(M),f(M'))$. The inverse construction is immediate.
 \end{proof}

 The exact structure on $\ld\bW\El$ induces an exact structure on $\bA\Pmd$. The \emph{conflations} 
 in $\bA\Pmd$ are usual short exact sequences, the \emph{deflations} are usual epimorphisms and the
 \emph{inflations} are such monomorphisms $\al:M\to N$ that both restrictions
 $\al_i=\al|_{e_iM}:e_iM\to e_iN\ (i=1,2)$ split. Suppose now that $\bW$ is projective and finitely generated
 as left $\bA_1$-module. Then the category $\bA\Pmd$, hence $\ld\bW\El$, contains enough
 projective objects. Indeed, in this case $\bA$ itself belongs to $\bA\pmd$, thus $\bA\Pro\sbe\bA\Pmd$
 and a usual projective resolution of a module $M\in\bA\Pmd$ is in fact its projective resolution in $\bA\Pmd$.
 One easily sees that a module $(P_1,\phi,P_2)$ is projective \iff $\phi$ is a split monomorphism.  Moreover,
 the following result holds.
 
 \begin{prop}\label{p12}
 Let $\bW$ be projective and finitely generated as left $\bA_1$-module, $\bA=\bA_1[\bW]\bA_2$. 
  \begin{enumerate}
  \item  $\pd M\le1$ for every $M\in\bA\Pmd$.
  \item  If $\,\gd\bA_i\le n$ for both $i=1,2$, then $\,\gd\bA\le n+1$.
  \end{enumerate}
 \end{prop} 
 \begin{proof}
 (1) \  Let $M=(P_1,\phi,P_2)\in\bA\Pmd$. Choose an epimorphism $P'\to\cok\phi$,
 where $P'\in\bA_1\Pro$, and lift it to a homomorphism $P'\to P_1$. Then we get a projective
 resolution of $M$:
  \begin{multline}\label{e11}
  0\larr (\bW\*_{\bA}P_2,0,0)\xarr{\Big(\mtr{-\phi\\\1},\ 0\Big)}\\
         \larr (P_1,0,0)\+(\bW\*_{\bA_2}P_2,\1,P_2)
	\xarr{\big((\1\ \phi)\,,\ \1\big)} (P_1,\phi,P_2)\larr 0. 
 \end{multline}
 We call \eqref{e11} the \emph{standard resolution} of $M$.
 
 \smallskip
 (2) \  Let $N$ be any $\bA$-module. Consider an exact sequence 
 \begin{equation}\label{e12}
  P(n-1)\xarr\al P(n-2) \to\dots \to P(1) \to P(0)\to N\to 0
 \end{equation}
 with $P(k)\in\bA\Pro$. It remains exact after multiplication by $e_i$ ($i=1,2$) Since $\gd\bA_i\le n$,
 $\Ker\al\in\bA\Pmd$, hence it has a projective resolution $0\to P(n+1)\to P(n)\to \Ker\al\to0$.
 Combiming it with the sequence \eqref{e12}, we get a projective resolution of length $n+1$ for $N$.
 \end{proof} 

 \begin{corol}\label{c13}
  Let $\bW$ be projective as $\bA_1$-module. Then $\cD^-(\ld\bW\El)\simeq\cD^-(\bA\Md)$
 and $\cD^b(\ld\bW\El)\simeq\cD\per(\bA\Md)$. In particular, if $\gd\bA_i<\8$ for both $i=1,2$, then
 $\cD^b(\ld\bW\El)\simeq\cD^b(\bA\Md)$.
 \end{corol} 

 Recall that $\cD\per(\bA\Md)$ denotes the full subcategory of $\cD(\bA\Md)$ consisting of two-side
 bounded complexes of projective modules (and isomorphic objects). Actually, it is equivalent to
 $\cH^b(\bA\Pro)$. If (and only if) $\gd\bA<\8$, then $\cD\per(\bA\Md)=\cD^b(\bA\Md)$.

 \begin{proof}
 The first equivalence follows from the known fact that both categories are equivalent to $\cH^-(\bA\Pro)$.
 Moreover, since all modules in $\bA\Pmd$ are of finite projective dimension,
 $\cD^b(\bW)\simeq \cH^b(\bA\Pro) \simeq\cD\per(\bA\Md)$.
 \end{proof}

 \begin{corol}\label{c14}
  Let both rings $\bA_i\ (i=1,2)$ be left noetherian and $\bW$ be finitely generated projective as
 $\bA_1$-module. Then $\cD^-(\ld\bW\el)\simeq\cD^-(\bA\md)$
 and $\cD^b(\ld\bW\el)\simeq\cD\per(\bA\md)$. In particular, if $\,\gd\bA_i<\8$ for both $i=1,2$, then
 $\cD^b(\ld\bW\el)\simeq\cD^b(\bA\md)$. 
 \end{corol}
 \begin{proof}
 Under these conditions the ring $\bA$ is also left noetherian, $\bA\pro\sbe\bA\pmd\sbe\bA\md$
 and $\bA\pmd\simeq\ld\bW\el$. Therefore, the arguments of the preceding proof can be applied.
 \end{proof}

 Note that the dual to the category $\bW\el$ is equivalent to the category $\bW\op\el$, where $\bW\op$
 denotes $\bW$ considered as $\bA\op$-bimodule. Namely, if $P\in\bA\pro$, set $\hat P=
 \Hom_\bA(P,\bA)\in\bA\op\pro$. Then 
 \[
 \bW(P)=\Hom_\bA(P,\bW\*_\bA P)\simeq \hat P\*_\bA \bW\*_\bA P
 \simeq \Hom_{\bA\op}(\hat P, \bW\op\*_{\bA\op}\hat P),
 \]
  and if we denote by $w\op$ the image of $w$ under this natural isomorphism, then
 $\Hom_\bW(w,v)\simeq \Hom_{\bW\op}(v\op,w\op)$. Obviously, this duality
 maps inflations to deflations and vice versa, so it is compatible with the exact structure.

 \section{Deformations in bimodule categories}
 \label{s2}

 From now on all rings are assumed to be finite dimesional algebras over an algebraically closed field $\Mk$;
 all bimodules $\bW$ are \emph{$\Mk$-central}, i.e. such that $\la w=w\la$ for every
 element $w\in\bW$ and every scalar $\la\in\Mk$, and finite dimensional over $\Mk$. In this case
 all vector spaces $\Hom_\bW(w,w')$, where $w,w'\in\bW\el$, are also finite dimensional. It
 implies that the Krull--Schmidt theorem on unique decomposition holds in this category. For every
 $P\!\in\bA\pro$, $\,\bW(P)$ is an affine space over $\Mk$. The algebraic group $\gl(P,\bA)$ acts
 regularly on $\bW(P)$: $g\cdot w=(1\*g)wg^{-1}$, and its orbits coincide with isomorphism classes of
 elements from $\bW(P)$ as objects of $\bW\el$. Therefore, one can speak about \emph{deformations}
 and \emph{degenerations} of such elements. Namely, we say that $w'$ is a \emph{degeneration} of $w$
 if $w'\in\ol{\gl(P,\bA)w}$ (the orbit closure in Zariski topology); in this case we write $w\pe w'$. It is useful to
 consider $\bW(P)$ as a $\Mk$-scheme; the set $\bW(P)(R)$ of its points in a $\Mk$-algebra $R$
 is
 \[
 (\bW\*R)(P\*R)\simeq \Hom_{\bA\*R}(P\*R,\bW\*_{\bA}(P\*R)) 
 \simeq \bW(P)\*R,
 \]
 where $\*$ denotes $\*_\Mk$. If $R\sbe R'$, there is a natural embedding
 $\bW(P)(R)\to\bW(P)(R')$ and we identify the elements of $\bW(P)(R)$ with their
 images in $\bW(P)(R')$. The following result is an analogue of  the well-known
 Riedmann--Zwara theorem on degenerations for finite dimensional modules \cite{rie,zw}.
 
 \begin{theorem}\label{t21}
 Let $w,w'\in\bW(P)$. The following conditions are equivalent:
 \begin{enumerate}
 \item  $w\pe w'$.
 \item  There is an object $v\in \bW\el$ and a conflation $w'\to w\+v\to v$.
 \item  There is an object $v\in \bW\el$ and a conflation $v\to w\+v\to w'$.
 \end{enumerate}
 \end{theorem}
 \begin{proof}
 Since $(\bW\el)\op\simeq\bW\op\el$, and this duality is compatible with the exact
 structures, we only need to prove that (1) and (2) are equivalent.
 We follow the original proofs of Riedtmann \cite{rie} and Zwara \cite{zw}.

 \smallskip
 (1)$\Arr$(2). \ Let $w\pe w'$.
 Then there is a discrete valuation algebra $R$ with the maximal ideal $\gM=tR$,
 the residue field $R/\gM\simeq\Mk$, the field of fractions $K$, and an element
 $W\in\bW(P)(R)$ such that
  \begin{align*}
  W&\simeq w\ \text{ in } \bW(P)(K),\\
 W\hskip-2ex\mod\gM&\simeq w'\ \text{ in } \bW(P)\simeq \bW(R)/\gM\bW(R).
 \end{align*}
 (cf. the proof of \cite[Theorem~1.2]{gh}). Choose an isomorphism $\ti\phi:w\to W$
 in $\bW(P)(K)$, i.e. an automorphism $\ti\phi:P\*K\to P\*K$ such that
 $(1\*\ti\phi)w=W\ti\phi$. Then $\ti\phi=t^{-r_1}\phi$ and
 $\ti\phi^{-1}=t^{-r_2}\phi'$ for some endomorphisms
 $\phi,\phi':P\*R\to P\*R$ and some integers $r_1,r_2$. Thus
 $\phi\phi'=\phi'\phi=t^r\1$ for $r=r_1+r_2$. Obviously, $\phi\in\Hom_\bW(w,W)$
 and $\phi'\in\Hom_\bW(W,w)$. Let $R_m=R/\gM^m$.
 Denote by $W_m$ and $w_m$ repectively the images of $W$ and $w$ in
 $\bW(P)(R_m)$. Since $R_m$ is an $m$-dimensional vector space,
 $P\*R_m\simeq mP$ as $\bA$-module, so the elements $W_m$ and $w_m$
 can be considered as objects from $\bW\el$. Moreover, $W_1\simeq w'$
 and $w_m\simeq mw$. The natural exact sequence $0\to \Mk\simeq R_1
\to R_{m+1} \to R_m\to 0$ induces a conflation $w'\simeq W_1\to W_{m+1}
 \to W_m$. Thus, we have only to show that $W_{m+1}\simeq w\+W_m$ for
 some $m$.

 There is an integer $k$ and a homomorphism of $\bA$-modules $\psi:P\*R\to P\*R$
 such that $\psi\phi(t^kx)=t^kx$ for every $x\in P\*R$ (see \cite[Lemma~3.4]{zw}), or,
 the same, $(\psi t^k)\phi=t^k\1$. Then
  \begin{align*}
  w\psi t^{k+r}&=w\psi t^k\phi\phi'=t^kw\phi'=t^k(1\*\phi')W=\\
		&=(1\*\psi t^k)(1\*\phi)(1\*\phi')W=(1\*\psi t^{k+r})W,
 \end{align*}
 so $\th=\psi t^{k+r}\in\Hom_\bW(W,w)$ and $\th\phi=t^{k+r}\1$. Let $Z$ be
 a complement of $\gM^{k+r}$ in $R$. Consider the $\bA$-homomorphism
 $q:P\*R\to P\*R$  that maps $p\*x\mps x$ for $x\in P\*\gM^{k+r}$ and
 $p\*z\mps 0$ for $z\in Z$. One immediately sees that $q\in\Hom_\bW(w,w)$
 and $q\th\phi=\1$. Therefore, $w$ is a direct summand of $W$ (in $\bW\El$).
 More precisely, $P\*R=X\+Y$, where $X=\im\phi,\ Y=\Ker q\th$, and
 with respect to this decomposition $W=\mtr{x&0\\0&y}$
 for some $y\in\bW(Y)$, where $x\simeq w$ in $(\bW\*R)\el$. Note that,
 since $\phi$ is a monomorphism of free $R$-modules of finite rank,
 $\dim_\Mk Y<\8$, so $y\in\bW\el$, and
 there is an integer $m$ such that $P\*\gM^m\sbe X$, so $P\*\gM^m\cap Y=0$.
 Therefore, $W_m\simeq x_m\+y$ and $W_{m+1}\simeq x_{m+1}\+y$. Since
 $x_{m+1}\simeq w_{m+1} \simeq w\+ w_m\simeq w\+x_m$, we get
 that $W_{m+1}\simeq w\+W_m$.

 \smallskip
 (2)$\Arr$(1). \ Suppose now that there is a conflation $w'\xarr\al w\+v\xarr\be v$,
 where  $w,w'\in\bW(P),\ v\in\bW(Q)$. Then the sequence 
 $  0\to P\xarr\al P\+Q \xarr\be Q\to 0 $
 is exact and $w\al=(\al\*1)w',\ v\be=(1\*\be)w$. If $\al=\mtr{\al_1\\\al_2},\,
 \be=(\be_1,\, \be_2)$, it means that $\al_1\in\Hom_\bW(w',w)$, 
 $\al_2\in\Hom_\bW(w',v)$, $\be_1\in\Hom_\bW(w,v)$ and $\be_2\in\Hom_\bW(v,v)$.
 Then $\be_\la=(\be_1,\, \be_2-\la\1_Q)$, where $\la\in\Mk$, also belongs to
 $\Hom_\bW(w\+v,v)$. Let $U=\setsuch{\la\in\Mk}{\be_\la \text{ is surjective}}$.
 It is an open neighbourhood of $0$ in $\Mk$. Moreover, there is an open neighbourhood
 $V\sbe U$ of $0$ and regular maps $V\to\Hom_\bA(P,P\+Q),\, \la\mps\al_\la$,
 and $\al':V\to\Hom_\bA(P\+Q,P),\, \la\mps \al'_\la$, such that
 $\al_\la=\Ker\be_\la$ and $\al'_\la\al_\la=\1_P$ for all $\la\in V$. Then, for
 $w_\la=(1\*\al'_\la)(w\+v)\al_\la$, the sequence $w_\la\xarr{\al_\la} w\+v\xarr{\be_\la} v$
 is a conflation. If $\be_2-\la\1_Q$ is invertible, it implies that $w_\la\simeq w$, and
 almost all $\la\in V$ satisfy this condition. On the other hand, $w_0\simeq w'$.
 Thus $w\pe w'$.
 \end{proof}

 \begin{corol}\label{c22}
 \begin{enumerate}
 \item   If $w\pe w'$, $\,\dim\Hom_\bW(w,z)\le\dim\Hom_\bW(w',z)$
 and $\dim\Hom_\bW(z,w)\le\dim\Hom_\bW(z,w')$ for all $z\in\bW\el$.
  \item  If $w\pe w'$ and either $\dim\Hom_\bW(w,w')=\dim\Hom_\bW(w',w')$
 or $\dim\Hom_\bW(z,w)=\dim\Hom_\bW(z,w')$ for all $z$, then
 $w\simeq w'$. Especially, if $\dim\Hom_\bW(w',w')=\dim\Hom_\bW(w,w)$,
 then $w\simeq w'$.
 \item   If there is a non-split conflation $\xi:\, w'\xarr\al w\xarr\be w''$, then
 $w\pe w'\+w''$ and $\dim\End_\bW(w)<\dim\End_\bW(w'\+w'')$.
 \end{enumerate}
 \end{corol} 
 \begin{proof}
 (1) \ A conflation $w'\xarr\al w\+v\xarr\be v$ induces exact sequences
 \begin{align}\label{e21}
  0\to \Hom_\bW(v,z) \xarr{\be^*} \Hom_\bW(w,z)\+\Hom_\bW(v,z)
	 \xarr{\al^*} \Hom_\bW(w',z)\\
 \intertext{and}
  0\to \Hom_\bW(z,w') \xarr{\al_*} \Hom_\bW(z,w)\+\Hom_\bW(z,v)
	 \xarr{\be_*} \Hom_\bW(z,v),
 \tag{2.1$'$}  
 \end{align}
 which implies the necessary inequalities.

 \smallskip
 (2) \ If the first equality holds, then the map $\al^*$ in the exact sequence \eqref{e21}
 for $z=w'$ is surjective, hence, there is a morphism $\al':w\+v\to w'$ such
 that $\al'\al=\1$, so $w\+v\simeq w'\+v$ and $w\simeq w'$. We get the same
 if the second equality holds for $z=v$. At last, if $w\not\simeq w'$, then
 $\dim\Hom_\bW(w,w)\le\dim\Hom_\bW(w,w')<\dim\Hom_\bW(w',w')$.

 \smallskip
 (3) \  In this case there is a conflation $\xi':\, w'\+w''\xarr{\al\+\1_{w''}} w\+w''
 \xarr{(\be \ 0)} w''$, so $w\pe w'\+w''$. If $\dim\End_\bW(w)=
 \dim\End_\bW(w'\+w'')$, then, as we have seen in the proof of (2), the conflation $\xi'$
 splits. Then $\xi$ splits as well.
 \end{proof}

 Suppose now that $\bW$ is an $\bA_1\df\bA_2$-bimodule, $\bA=\bA_1[\bW]\bA_2$
 and consider the equivalence $\ld\bW\el\simeq \bA\pmd$, $w\mps M(w)$.
 We denote by $\rpp(P_1,P_2)$ the set of all $\bA$-modules $M$ such that $e_iM
 \simeq P_i\ (i=1,2)$ or, the same, $M\simeq M(w)$ for $w\in\ld\bW(P_2,P_1)$.
 Let $\lst Ss$ be the set of isomorphism classes of simple $\bA$-modules.
 Recall that the \vd\ $\Dim M$ of an $\bA$-module $M$ is defined as the
 integer vector $\row ds$, where $d_k=(M:S_k)$, the multiplicity of $S_k$ in a
 Jordan--H\"older series of $M$. We denote by $\rep(\fD)$ (or $\rep(\fD,\bA)$, if
 necessary), the variety of $\bA$-modules of vector dimension $\fD$. Then
 $\rpp(P_1,P_2)$ is a subvariety of $\rep(\fD)$ with $\fD=(\fD_1,\fD_2)$,
 where $\fD_i=\Dim P_i$. Obviously, $w\pe w'$ \iff $M(w)\pe M(w')$.
 
 \begin{prop}\label{p23}
 $\rpp(P_1,P_2)$ is open in $\rep(\fD)$ and its closure $\rep(P_1,P_2)$ is an
 irreducible component of $\rep(\fD)$.
 \end{prop}
 \begin{proof}
 Let $M\in\rpp(P_1,P_2)$, $\pi_i:\rep(\fD,\bA)\to\rep(\fD_i,\bA_i)$ be the natural
 projections.  As projective modules are \emph{rigid} \cite{gab}, there are neighbourhoods
 $U_i\sbe\rep(\fD_i,\bA_i)$ of $P_i$ such that $N\simeq P_i$ for every $N\in U_i,
 (i=1,2)$. Then $\pi^{-1}_1(U_1)\cap\pi^{-1}_2(U_2)$ is a neighbourhood of $M$
 contained in $\rpp(P_1,P_2)$. Since $\rpp(P_1,P_2)\simeq\mA^m$, where
 $m=\dim\ld\bW(P_2,P_1)$, and hence it is irreducible, its closure in $\rep(\fD)$ is an irreducible
 component.
 \end{proof} 

 Dealing with derived equivalences, it is more convenient to consider projective resolutions
 instead of modules. A degeneration theory for complexes of projective modules has been
 developed in \cite{jsz}. In particular, an analogue of the Riedtmann--Zwara theorem for
 complexes has been proved there. Let $\lst\Pi s$ be the set of isomorphism classes of
 indecomposable projective $\bA$-modules. Then every projective $\bA$-module is
 isomoprhic to a module $P\row ps=\bop_{k=1}^s p_k\Pi_k$ for a uniquely defined integer
 vector $\row ps\in\mN^s$. Let $\fP:n\mps\fP_n$ be a function $\mZ\to\mN^s$.
 We denote by $\cP(\fP)$ (or by $\cP(\bA,\fP)$, if nesessary) the set of all complexes
 $\kP=(\kP_n,d_n)$ such that $\kP_n=P(\fP_n)$. If $\fP$ is bounded, i.e. $\fP_n=0$ for
 almost all $n$, the set $\cP(\fP)$ is an algebraic variety, a closed subvariety of
 $\prod_n\Hom_\bA(\kP_n,\kP_{n-1})$. Isomorphism classes of complexes from $\cP(\fP)$
 are orbits of the group $\gl(\fP)=\prod_n\gl(\kP_n,\bA)$. Thus, the relation $\pe$ is
 well defined for complexes from $\cP(\fP)$: $\kP\pe\kP'$ if $\kP'\in\ol{\gl(\fP)\kP}$.
 The following result is a special case of \cite[Theorems~1,2]{jsz}.
 
 \begin{theorem}\label{t24}
  Let $\fP:\mZ\to\mN^s$ be bounded, $\kP,\kP'\in\cP(\fP)$. Then $\kP\pe\kP'$ \iff
 there is a complex $\kQ\in\cH^-(\bA\pro)$ and a triangle 
 \[
   \kP'\to\kP\+\kQ\to\kQ\to\kP'[1].
 \]
 If $\gd\bA<\8$, one can choose $\kQ$ bounded.
 \end{theorem} 

 \begin{corol}\label{c25}
  Let $\kP,\kP'$ are two complexes from $\cP(\fP)$, where $\fP$ is bounded, such that
 $\Dim\rH_k(\kP)=\Dim\rH_k(\kP')$ for all $k$.
 \begin{enumerate}
 \item  If $\kP\pe \kP'$, then $\rH_k(\kP)\pe \rH_k(\kP')$ for all $k$.
 \item  If $\bA$ is \emph{hereditary}, or, the same, Morita equivelent to a path algebra of
 a quiver \cite{dk}, and $\rH_k(\kP)\pe \rH_k(\kP')$ for all $k$, then $\kP\pe \kP'$.
 \end{enumerate}
 \end{corol}
 \begin{proof}
 (1) \ We may suppose that $\fP_k=0$ for $k<0$. Consider a triangle
 $\kP'\xarr\al\kP\+\kQ\xarr\be\kQ$. It induces an exact sequence of homologies
 \begin{align*}
  \cdots& \to \rH_2(\kP') \xarr{\rH_2(\al)} \rH_2(\kP)\+\rH_2(\kQ)
			 \xarr{\rH_2(\be)} \rH_2(\kQ) \to \\
   &\to \rH_1(\kP') \xarr{\rH_1(\al)} \rH_1(\kP)\+\rH_1(\kQ)
			 \xarr{\rH_1(\be)} \rH_1(\kQ) \to \\
   &\to \rH_0(\kP') \xarr{\rH_0(\al)} \rH_0(\kP)\+\rH_0(\kQ)
			 \xarr{\rH_0(\be)} \rH_0(\kQ) \to 0
 \end{align*}
 Since $\Dim\rH_0(\kP)=\Dim\rH_0(\kP')$, the map $\rH_0(\al)$ is injective, hence,
 $\rH_1(\be)$ is surjective. Now, since $\Dim\rH_1(\kP)=\Dim\rH_1(\kP')$, the map
 $\rH_1(\al)$ is injective, hence, $\rH_0(\be)$ is surjective, etc. Thus we get exact
 sequences
 \[
 0\to\rH_k(\kP')\to\rH_k(\kP)\+\rH_k(\kQ)\to\rH_k(\kQ)\to0  
 \]
  for all $k$, whence the statement follows.

\smallskip
 (2) \ Note that for projective modules $\kP,\kP'$ over a hereditary algebra $\dim\kP=\dim\kP'$
 \iff $\kP\simeq\kP'$. Recall also that every bounded complex $\kP$ of projective modules over
 a hereditary algebra is isomorphic to the direct sum $\bop_k\kP^{(k)}[k]$, where $\kP^{(k)}$ is the complex
 \[
 0\to\im d_{k+1}\xarr{\bar d_{k+1}}\Ker d_k\to 0\ \text{ ($\Ker d_k$ at the $0$-th place),}  
 \]
  where $\bar d_{k+1}$ is injective and $\rH_0(\kP^{(k)})\simeq\rH_k(\kP)$. If $\rH_k(\kP)\pe\rH_k(\kP')$,
 there is an exact sequence
 \[
 0\to\rH_k(\kP')\to\rH_k(\kP)\+N\to N\to 0,  
 \]
  which induces an exact sequenece of complexes
 \[
 0\to {\kP'}^{(k)}\to \ti\kP \to \kQ\to 0,  
 \]
  where $\kQ$ is a projective resolution of $N$, $\ti\kP_i={\kP'}^{(k)}_i\+\kQ_i$
 and $\ti\kP$ is a projective resolution of
 $\rH_k(\kP)\+N$. Therefore, $\ti\kP$ is homotopy equivalent to $\kP^{(k)}\+\kQ$ and, since
 all their components are of the same dimensions, they are isomorphic (see \cite[Lemma~1]{jsz}).
 Thus ${\kP'}^{(k)}\pe\kP^{(k)}$ for each $k$, which evidently implies that $\kP'\pe\kP$.
 \end{proof}

 We denote by $\Hom_{\cD\bA}(\kC,\kC')$ the set of morphisms $\kC\to\kC'$ in the
 derived category $\cD(\bA\Md)$. Quite analogously to Corollary~\ref{c22}, one proves

 \begin{corol}\label{c26}
 \begin{enumerate}
 \item   If $\,\kP\pe \kP'$\!, 
 $\,\dim\Hom_{\cD\bA}(\kP,\kQ)\le\dim\Hom_{\cD\bA}(\kP',\kQ)$
 and  $\dim\Hom_{\cD\bA}(\kQ,\kP)\le\dim\Hom_{\cD\bA}(\kQ,\kP')$ for all
 $\kQ\in\cD(\bA\md)$.
  \item  If $\kP\pe \kP'$ and either
 $\dim\Hom_{\cD\bA}(\kP,\kP')=\dim\Hom_{\cD\bA}(\kP',\kP')$ or
 $\dim\Hom_{\cD\bA}(\kQ,\kP)=\dim\Hom_{\cD\bA}(\kQ,\kP')$
 for all $\kQ$, then $\kP\simeq \kP'$.
 \item   If there is a non-split triangle $\kP'\to\kP\to\kP''\to\kP[1]$, then
 $\kP\pe\kP'\+\kP''$ and
 $\dim\End_{\cD\bA}(\kP)<\dim\End_{\cD\bA}(\kP'\+\kP'')$.
 \end{enumerate}
 \end{corol}

 \section{Tits form and Euler form}
 \label{s3}

 Let $\bW$ be an $\bA$-bimodule. We define the \emph{Tits form} of $\bW$ as the
 bilinear form on the Grothendieck group $K_0(\bA\pro)$ of projective $\bA$-modules 
 \[
  \gnr{[P],[P']}_\bW=\dim\Hom_\bA(P,P')-\dim\Hom_\bA(P,\bW\*_\bA P'). 
 \]
 The corresponding quadratic form $\rQ_\bW([P])=\gnr{[P],[P]}_\bW$ is the classical
 Tits form, which equals the dimension of the group $\gl(P,\bA)$ minus the
 dimension of the affine space $\bW(P)$, where this group acts so that the orbits are
 the isomorphism classes. 

 If $\bW$ is a bipartite $\bA_2\df\bA_1$-bimodule, $P=P_1\+P_2$ and $P'=P_1'\+P'_2$,
 where $P_i,P'_i$ are $\bA_i$-modules, then
 \[
  \gnr{[P],[P']}_\bW=\dim\Hom_{\bA_1}(P_1,P'_1)+\dim\Hom_{\bA_2}(P_2,P_2')
	-\dim\bW(P_2,P'_1).
 \]
 In particular, if $\bA=\bA_1[\bW]\bA_2$ and we consider the Tits form of the dual
 bimodule $\ld\bW$, then 
 \[
  \rQ_{\ld\bW}(P_1\+P_2)=\dim\gl(P_1,\bA_1)+\dim\gl(P_2,\bA_2)-
	\dim\rep(P_1,P_2).
 \]

 If, moreover, $\bW$ is projective as left $\bA_1$-bimodule and both $\bA_i$ are of finite
 global dimension, the algebra $\bA=\bA_1[\bW]\bA_2$ is also of finite global dimension,
 so the Grothendieck group $K_0(\bA\pro)$ coincides with the Grothendieck group
 $K_0(\bA\md)$ of all $\bA$-modules. In this case the \emph{Euler form} on this group
 can be defined as
 \[
  \gnr{[M],[M']}_\bA=\sum_k (-1)^k\dim\ext^k_\bA(M,M'). 
 \]
 In particular, if $P$ and $P'$ are projective, $\gnr{[P,[P']}_\bA=\dim\Hom_\bA(P,P')$.
 
 \begin{prop}\label{p31}
 Let $\bW$ be an $\bA_1\df\bA_2$-bimodule projective as left $\bA_1$-module,
 $M\in\rpp(P_1,P_2)$, $M'\in\rpp(P'_1,P'_2)$. Then
 $\gnr{[M],[M']}_\bA=\gnr{[P],[P']}_{\ld\bW}$, where
 $P=P_1\+P_2,\,P'=P_1'\+P_2'$.
 \end{prop}
 \begin{proof}
 Consider the standard resolutions \eqref{e11} $0\to\kQ\to\kP\to M\to0$ of $M$ and
 $0\to\kQ'\to\kP'\to M'\to0$ of $M'$. Then $[M]=[\kP]-[\kQ]$ and $[M']=[\kP']-[\kQ']$
 in $K_0(\bA\md)$. Therefore,
 \begin{align*}
  &\gnr{[M],[M']}_\bA=\gnr{[\kP],[\kP']}_\bA+\gnr{[\kQ],[\kQ']}_\bA
	-\gnr{[\kP],[\kQ']}_\bA-\gnr{[\kQ],[\kP']}_\bA = \\
	&= \dim\Hom_{\bA_1}(P_1,P_1')+\dim\Hom_{\bA_2}(P_2,P_2') +
	\dim\Hom_{\bA_1}(P_1,\bW\*_{\bA_1}P_2')+\\
	&+\dim\Hom_{\bA_1}(\bW\*_{\bA_2}P_2,\bW\*_{\bA_2}P_2')
	-  \dim\Hom_{\bA_1}(P_1,\bW\*_{\bA_2}P'_2) -\\
	&- \dim\Hom_{\bA_1}(\bW\*_{\bA_2}P_2,P'_1)
	 -\dim\Hom_{\bA_1}(\bW\*_{\bA_2}P_2,\bW\*_{\bA_2}P'_2) =\\
    	&= \dim\Hom_{\bA_1}(P_1,P_1')+\dim\Hom_{\bA_2}(P_2,P_2') -
	\dim\Hom_{\bA_1}(\bW\*_{\bA_2}P_2,P'_1) =\\
	& = \gnr{[P],[P']}_{\ld\bW}, \text{ since }
	 \Hom_{\bA_1}(\bW\*_{\bA_2}P_2,P'_1)\simeq \ld\bW(P_2,P'_1).
 \end{align*}
\vskip-\baselineskip
 \end{proof}

 \begin{corol}\label{c32}
  Under conditions of Proposition~\ref{p31}, the codimension of the orbit of $M$ in
 $\rep(P_1,P_2)$ equals $\dim\ext^1_\bA(M,M)$. In particular, this orbit is open in
 $\rep(P_1,P_2)$ \iff $\ext^1_\bA(M,M)=0$. 
 \end{corol}
 \begin{proof}
 This codimension equals 
  \begin{align*}
  &\dim\rep(P_1,P_2)-(\dim\gl(P_1,A_1)+\dim\gl(P_2,A_2)-\dim\Aut_\bA M)=\\
  & =\dim\ld\bW(P_1,P_2)-\dim\gl(P_1,A_1)-\dim\gl(P_2,A_2)+\dim\End_\bA M=\\
  & =-\gnr{[P],[P]}_{\ld\bW} +\dim\End_\bA(M,M)=\dim\ext^1(M,M),
 \end{align*}
 since $\pd M\le1$ and
 \[
  \gnr{[P],[P]}_{\ld\bW}=\gnr{[M],[M]}_\bA= \dim\End_\bA M-\dim\ext^1_\bA(M,M).
 \]
 \end{proof}

 Recall that analogous results hold for representations of quivers, or, the same, for modules
 over the path algebra $\bB=\Mk\De$ of a quiver $\De$ \cite{rin}. If $\De_0$ is the set of vertices
 and $\De_1$ is the set of arrows of $\De$, the \emph{Tits form} is defined as
 \[
  \gnr{[M],[M']}_\De=\sum_{i\in\De_0}d_id'_i-\sum_{a\in\De_1}d_{\si a}d'_{\tau a}, 
 \]
 where, for any representation $M$ of the quiver $\De$, $d_i=\dim M(i)$ (thus
 $\row ds=\Dim M$, the dimension vector of $M$), $\si a$ is the \emph{source} of the arrow
 $a$ and $\tau a$ is its \emph{target}, i.e. $a:\si a\to \tau a$. The corresponding quadratic form
 is again of geometric meaning: if $\fD=\dim M=\row ds$, then
 \[
  \rQ_\De(\fD)=\gnr{[M],[M]}_\De=\dim\gl(\fD,\Mk)-\dim\rep(\fD,\bB),
 \]
 where $\gl(\fD,\Mk)=\prod_i\gl(d_i,\Mk)$.
 The Tits form coincides with the \emph{Euler form}: source
  \[
   \gnr{[M],[M']}_\De=\gnr{[M],[M']}_\bB=
  \dim\Hom_\bB(M,M')-\dim\ext^1_\bB(M,M'),
 \]
 and the codimenion in $\rep(\fD,\bB)$ of the orbit of the representation $[M]$ equals
 $\dim\ext^1_\bB(M,M)$.
 
 \section{Tilting and deformations}
 \label{s4}

 Let $\bA$ and $\bB$ be derived equivalent finite dimensional $\Mk$-algebras. We suppose
 that they are of finite global dimension. Then it follows from \cite{ric} (see also \cite{kel})
 that there is a bounded complex $\kT=(\kT_n,\dd_n)$ of finite dimensional projective
 $\bB\df\bA$-bimodules such that the functor $\kT\tl_\bA\hskip-.7ex\_\,$
 is an equivalence $\cD(\bA\md)\to\cD(\bB\md)$ and $\rh_{\bB}(\kT,\_\,)$ is a
 quasi-inverse equivalence $\cD(\bB\md)\to\cD(\bA\md)$. We denote $\kT M=\kT\*_{\bA}M$,
$\kT_kM=\kT_k\*_{\bA}M$ and $\kH_kM=\rH_k(\kT M)$.
 
 \begin{lemma}\label{l41}
  Let $Z$ be a component of $\rep(\fD,\bA)$ for some vector dimension $\fD$. There
 is an open subset $Z^\kT\sbe Z$ such that for all $M\in Z^\kT$ and all $k$ the
 vector dimensions $\Dim\kH_k M$ are the same and the smallest (componentwise)
 among $\Dim\kH_kN$ for $N\in Z$. Moreover, if $M,M'\in Z^{\kT}$ and $M\pe M'$,
 then $\kH_kM\pe\kH_kM'$ for all $k$. 
 \end{lemma}
 \begin{proof}
 The first assertion is evident: take for $U$ the set of all $M\in Z$ such that 
 $\Dim\im(\dd_k\*1_M)$ is the biggest possible. If $M\pe M'$, Theorems~\ref{t21}
 and \ref{t24} show that $\kT M\pe\kT M'$, so the second assertion follows from
 Corollary~\ref{c25}.
 \end{proof}

 Thus, one can define a map $Z_\kT\to\rep(\fH_k,\bB)$, where $\fH_k=\Dim\kH_kM$.
 Moreover, for every $M\in Z_\kT$ there is an open neighbourhood $U\ni M$ in $Z^\kT$
 such that this map can be considsered as a regular map $U\to\rep(\fH_k,\bB)$. 

 Let also $Z_0$ be the \emph{open sheet} of $Z$, i.e. the set of all $M\in Z$ such that
 $\dim_\Mk\End_{\bA}M$ is minimal possible, or, the same, the dimension of the orbit
 of $M$ is maximal possible ($Z_0$ is always open in $Z$).
 We denote $Z^\kT_0=Z^\kT\cap Z_0$ and call $Z^\kT_0$ the \emph{$\kT$-sheet}
 of $Z$. Note that if $M'\in Z_0$ and $M\pe M'$, then $M\simeq M'$. In particular,
 if $M\simeq M_1\+M_2\in Z_0$, then $\ext^1(M_1,M_2)=0$. Therefore, the same
 is true for the object $\kT M$ from $\cD(\bB\md)$. Note that the group $\gl(\fD,\Mk)$ 
 acts on $Z_0$ and this action is closed. 
 
 We apply these considerations to the component $\rep(P_1,P_2)$ of $\rep(\bA)$,
 when $\bA=\bA_1[\bW]\bA_2$, getting the $\kT$-sheet $\rep^\kT_0(P_1,P_2)$
 and an open subset ${\rpp}^\kT_0(P_1,P_2)=\rpp(P_1,P_2)\cap\rep^\kT_0(P_1,P_2)$,
 which we call the \emph{$\kT$-sheet} of $\rpp(P_1,P_2)$. The isomorphism
 $\rpp(P_1,P_2)\simeq \ld\bW(P_2,P_1)$ maps it to an open subset
 $\ld\bW^\kT_0(P_1,P_2)$, also called the \emph{$\kT$-sheet} of $\ld\bW(P_2,P_1)$.

 Suppose now that the algebra $\bA=\bA_1[\bW]\bA_2$ is derived equivalent to a
 path algebra $\bB=\Mk\De$ of a quiver $\De$. Let $\fD$ be a dimension of representations
 of $\bB$. Then $\rep(\fD,\bB)$ is an affine space and there is an open subset
 $\rep^c(\fD)\sbe\rep(\fD,\bB)$ such that every $N\in\rep^c(\fD)$ decomposes
 as $N\simeq\bop_{i=1}^m N_i$, where the number of summands $m$ and
 $\dim N_i=\fD_i$ are common for all $N\in\rep^c(\fD)$, all modules $N_i$
 are \emph{bricks} (or \emph{schurian}), i.e. $\End_\bA N_i=\Mk$, and
 $\ext^1_\bA(N_i,N_j)=0$ for $i\ne j$ \cite[Proposition 3]{kac}. Recall also
 that every object $\kC\in\cD^b(\bB\md)$ is isomorphic to the direct sum
 of shifted modules $\bop_kH_k[k]$, where $H_k=\rH_k(\kC)$. Especially, it is the
 case for $\kT M$, where $M\in\bA\md$: it is isomorphic to $\bop_k\kH_kM[k]$.
 
  \begin{lemma}\label{l42}
 Let $\bA$ be a Dynkinian algebra, $Z_0$ be the open sheet of an irreducible component $Z$
 of the variety $\rep(\fD,\bA)$ for some dimenion $\fD$. Then $Z_0$ consists of a unique
 isomorphim class of modules, and $\ext^1_\bA(M,M)=0$ for any $M\in Z_0$.
 \end{lemma}
 \begin{proof}
 It is known that if $\bB=\Mk\De$, where $\De$ is a Dynkin quiver, $N$ is an indecomposable
 $\bB$-module, then $\ext^1_\bB(N,N)=0$ \cite{rib}. Therefore, it also holds for
 any indecomposbale object from $\cD^b(\bB\md)$, hence from $\cD^b(\bA\md)$ as well.
 In particular, $\ext^1_\bA(M_i,M_i)=0$ for every indecomposable direct summand of
 $M$. Since $\ext^1_\bA(M_i,M_j)=0$ for any two different direct summands,
 $\ext^1_\bA(M,M)=0$. Moreover, both $\bB$ and $\bA$ have finitely many
 isomorphism classes of indecomposable modules. Hence there are finitely many orbits
 in $Z$, so one of them is dense. Then it coincides with $Z_0$.
 \end{proof}

 Recall that an $\bA$-module $M$ is called \emph{partial tilting} \cite{hap} if
 $\pd M\le1$ and $\ext^i_\bA(M,M)=0$ for all $i\ne0$. If, moreover, $M$
 generates $\kD^b(\bA)$, it is called \emph{tilting}. If $M$ is tilting,
 $\ti\bA=(\End_\bA M)\op$, then $\cD^b(\ti\bA\md)\simeq\cD^b(\bA\md)$
 \cite[Theorem~III.2.10]{hap}. For any partial tilting module $M$ there is a module
 $M'$ such that $M\+M'$ is tilting \cite[Lemma~III.6.1]{hap}.

 \begin{corol}\label{c43}
  Let $\bA=\bA_1[\bW]\bA_2$, where $\bW$ is a Dynkinian bimodule projective as
 $\bA_1$-module and $w$ belongs to the open sheet of $\ld\bW(P_1,P_2)$. Then
 $\ext^i_{\ld\bW}(w,w)=0$ for all $i>0$. In particular, $\bA'=\End_{\ld\bW}w$ is
 a Dynkinian algebra.
 \end{corol}
 \begin{proof}
 The first claim follows from Lemma~\ref{l42} and Proposition~\ref{p12} applied to the
 module $M\in\bA\pmd$ corresponding to $w$. In particular, $M$ is a partial
 tilting module, thus is a direct summand of a tilting module $\ti M$.
 Then $\bA'\simeq\End_\bA M\simeq e\ti\bA e$, where
 $\ti\bA=\End_\bA M$ and $e$ is an idempotent from $\ti\bA$. Since
 $\cD^b(\ti\bA)\simeq\cD^b(\bA)$, $\ti\bA$ is a Dynkinian algebra,
 hence so is also $\bA'$.
 \end{proof}

 Let now $\bB=\Mk\De$, where $\De$ is a Euclidean quiver. Recall that a $\bB$-module $M$ of
 \vd\ $\fD=\row ds$, as well as this dimension itself, is called \emph{sincere},
 if $d_i\ne0$ for all $i$, or, equivalently, $\Hom_\bB(P,M)\ne0$ for every projective
 $\bB$-module $P$. In Euclidean case the quadratic Euler form
 $\rQ(\fD)=\gnr{\fD,\fD}_\bB$ is non-negative, i.e. $\rQ(\fD)\ge0$ for every vector
 $\fD$. Moreover, its kernel is one-dimensional, so there is a unique vector $\bd$ with
 coprime positive integer coordinates such that $\rQ(\bd)=0$. A non-negative integer
 vector $\fD$ is a dimension of an indecomposable $\bB$-module \iff either
 $\rQ(\fD)=1$ or $\rQ(\fD)=0$. In the former case $\fD$ is called a \emph{real root};
 then there is a unique indecomposable representation $M$ of \vd\ $\fD$. In the latter case
 $\fD$ is called an \emph{imaginary root}; there are infinitely many non-isomorphic
 indecomposable modules of \vd\ $\fD$. Any imaginary root is sincere and equals $m\bd$
 for some $m$. 

 Let $\tau$ denotes the \emph{Auslander--Reiten translation} in the category $\cD^b(\bB\md)$;
 it is an auto-equivalence of this category \cite{hap}. Then the category $\bB\ind$
 consists of three disjoint parts:
 
\smallskip
 \emph{preprojective part} $\cP$ consisting of modules $\tau^{-m}P\ (m>0)$, where
 $\tau$ is the Auslander--Reiten transform and $P$ runs through indecomposable projective
 $\bB$-modules;

\smallskip
 \emph{preinjective part}  $\cI$ consisting of modules $\tau^mI\ (m>0)$, where
 $I$ runs through indecomposable injective $\bB$-modules;

\smallskip
 \emph{regular part} $\cR$, which is a disjoint union of \emph{tubes} $\cR_\la$
 ($\la\in\mP^1_\Mk$), where $\cR_\la$ is equivalent to the category of finite
 dimensional representations $R$ of a cyclic quiver $\bH_n$:
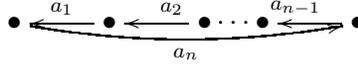
 \[
  \xymatrix{\bullet \ar@{<-}[r]^{a_1} & \bullet \ar@{<-}[r]^{a_2} & {\bullet \dots \bullet} 
 \ar@{<-}[r]^{a_{n-1}} & \bullet \ar@/^/@{<-}[lll]^{a_n} }
 \]
 such that $R(a_1a_2\dots a_n)$ is nilpotent. Here $n=n(\la)$ depends on $\la$ and equals $1$
 for all $\la\in\mP^1_\Mk$ except, possibly, $1,2$ or $3$ points. We denote by $\mX$
 the subset $\setsuch{\la\in\mP^1_\Mk}{n(\la)=1}$. Especially, every indecomposable module
 $M$ such that $\Dim M$ is an imaginary root belongs to $\cR$.

 More precisely, the indecomposable modules $M$ such that $\Dim M$ is an imaginary root can be
 described as follows \cite{df,naz}. There is a locally free sheaf $\kF$ over the projective line
 $\mP^1_\Mk$ with $\bB$-action on it such that the $\bB$-module $\kF(m,\la)=\kF_\la/\gM^m_\la\kF$,
 where $\la\in\mP^1_\Mk$ and $\gM_\la$ is the maximal ideal of $\kO_{\mP^1,\la}$,
 is indecomposable of dimension $m\bd$. Every  indecomposable module from $\cR_\la$
 for $\la\notin\mX$ is isomorphic to $\tau^k\kF(m,\la)$ for some $m$ and $0\le k<n(\la)$,
 in particular, if $\la\in\mX$, it is isomorphic to $\kF(m,\la)$.
 The map $\la\mps\tau^k\kF(m,\la)$ with fixed $k$ induces a regular embedding
 $\mP^1\to\rep(m\bd)$. Note also that $\kF(m,\la)$ is a brick \iff $m=1$.

\smallskip
 Moreover, $\Hom_\bB(M,N)=0$ if either $M\in\cI,\,N\in\cP\cup\cR$, or $M\in\cR,\,N\in\cP$,
 or $M\in\cR_\la,\,N\in\cR_\mu,\ \la\ne\mu$, while $\Hom_\bB(M,N)\ne0$ and
 $\ext^1_\bB(M,N)\ne0$ if either $M\in\cP,\,N\in\cR_\la$ with $\la\in\mX$, or
 $M\in\cR_\la,\,N\in\cI$ with $\la\in\mX$, or $M,N\in\cR_\la$ with $\la\in\mX$.
 
 Following \cite{hap}, we denote by $\cC[i]=\cP[i]\cup\cI[i+1]$; thus
 $\cD^b(\bB\Md)=\bup_i(\cC[i]\cup\cR[i])$, and if $M\in\cR_\la[i]$ with
 $\la\in\mX$, $N$ is an indecomposable object of $\cD^b(\bB)$, then
 \begin{align*}
   \Hom_{\cD\bB}(M,N)\ne0 &\ \text{ \iff either}\ N\in\cC[i+1] &&\\
			   	&\ \text{ or }\ N\in\cR_\la[j],\ j\in\{i,i+1\}; &&\\
  \Hom_{\cD\bB}(N,M)\ne0  &\ \text{ \iff either}\ N\in\cC[i] &&\\
			&\ \text{ or }\ N\in\cR_\la[j],\ j\in\{i,i-1\}.&&
 \end{align*}
 We also need the following simple lemma.

 \begin{lemma}\label{l44}
 Let $\bB$ be a hereditary algebra of Euclidean type, $M\in\bB\ind$ and $\fD=\dim M$.
 \begin{enumerate}
\item  If $M$ is not a brick or $\ext^1_\bB(M,M)\ne0$, then
 \begin{align*}
  \Hom_\bB(N,M)\ne0\ne\ext^1_\bB(M,N) &\ \text{ \em for }\ N\in\cP;&&\\
  \Hom_\bB(M,N)\ne0\ne\ext^1_\bB(N,M) &\ \text{ \em for }\ N\in\cI.&&
 \end{align*}
 \item  If $M$ is a brick and $\fD$ is a real root, then $\ext^1_\bB(M,M)=0$.
 \item  If $\dim M$ is an imaginary root, then for every $\bB$-module $N$ 
 \[
 \dim\Hom_\bB(M,N)+\dim\Hom_\bB(N,M)=\dim\ext^1_\bB(M,N)+
 \dim\ext^1_\bB(N,M).
 \]
\end{enumerate}
 \end{lemma}
 \begin{proof}
 (1) \ Since every proper subgraph of an Euclidean graph is a Dynkin one, every non-sincere
 $\bB$-module is a brick without self-extensions. Thus $M$ is sincere, hence
 $\Hom_\bB(P,M)\ne0$ for every projective $P$. Since $\tau^mM$ satisfies the
 same conditions, $\Hom_\bB(N,M)\ne0$ for every preprojective $N$. It remains to
 note that $\ext^1_\bB(M,N)\simeq\Hom_\bB(\tau^{-1}N,M)^*$, where $V^*$
 denotes the dual vector space to $V$ (see \cite[Proposition~1.4.10]{hap}),
 and use the duality to get the assertion for preinjective modules.

 (2) Follows from the formula $\rQ(\fD)=\dim\Hom_\bB(M,M)-\dim\ext^1_\bB(M,M)$.

 (3) \ Since $\dim M$ belongs to the kernel of the Euler form $\rQ_\bB$,
 \begin{align*}
  \gnr{M,N}_\bB+\gnr{N,M}_\bB&= \dim\Hom_\bB(M,N)+\dim\Hom_\bB(N,M)-\\
 	&-\dim\ext^1_\bB(M,N)-\dim\ext^1_\bB(N,M) =0.
 \end{align*}
 
 \end{proof}

  \begin{lemma}\label{l45}
 Let $f:X\to Y$ be a morphism of irreducible algebraic varieties, $\sG$ and $\sH$ be connected
 algebraic groups acting respectively on $X$ and $Y$ so that $\sG x=\sG x'$ \iff
 $\sH f(x)=\sH f(x')$. Suppose also that $\dim\stab x=\dim\stab f(x)$ and
 $\cod \sG x=\cod\sH f(x)$ for all $x\in X$. Then $\sH f(X)$ contains an open dense
 subset of $Y$.
 \end{lemma}
 \begin{proof}
  Obviously, we may replace $X$ by an open subset $X_0$ such that the dimension $\dim\stab x=s$
 is minimal possible for all points $x\in X_0$, hence, also $\cod \sG x=c$ is constant for $x\in X_0$.
  Choose an orbit $\sG x_0$ with $x_0\in X_0$ and a subvariety $Z\sbe X_0$ of dimension
 $c=\cod \sG x$ that intersects $\sG x_0$ in finitely many points. Then there is an open subset
 $X'\sbe X_0$ such that $Z\cap \sG x$ is finite for every $x\in X'$. Therefore $\sH y\cap f(Z)$
 is finite for every $y\in f(X')$, and $\dim f(Z)=\dim Z$. It implies that
 $\dim \sH f(Z)=\dim \sH -s+c=\dim \sH y+c=\dim Y$ for any $y\in f(Z)$. Hence
 $\sH f(Z)$ contains an open dense subset from $Y$.
 \end{proof}

  \begin{corol}\label{c46}
 Let $\bA=\bA_1[\bW]\bA_2$, where the bimodule $\bW$ is projective as left $\bA_1$-module,
 such that $\bA$ is derived equivalent to a quiver algebra $\bB=\Mk\De$, and $\kT$ be a bounded complex of
 projective $\bB\df\bA$-bimodules establishing this equivalence. Let also $M\in{\rpp}^\kT_0(P_1,P_2)$
 and $\Dim\kH_kM=\fD_k$. Then there are open dense subsets $U_k\sbe\rep(\fD_k)$
 such that for every set $\setsuch{N_k}{N_k\in U_k}$ there is a module $M'\in{\rpp}^\kT_0(P_1,P_2)$
 with $\kH_kM'\simeq N_k$.
 \end{corol}
 \begin{proof}
 Obviously, we may suppose that there is a unique $k$ such that $\fD_k\ne0$.
  We use Lemma~\ref{l45} for $X={\rpp}^\kT_0(P_1,P_2)$ with the natural action of
 $\gl(P_1,\bA_1)\xx\gl(P_2,\bA_2)$ and $Y=\rep(\fD_k)$ with the natural action
 of $\gl(\fD_k,\Mk)$. One only has to note that in both cases
 $\dim\stab M=\dim\Hom(M,M)$ and the codimension of the orbit equals $\dim\ext^1(M,M)$,
 so they are presereved under the equivalence of the derived categtories.
 \end{proof}

 In particular, suppose that $\bB$ is tame, $M\in{\rpp}^\kT_0(P_1,P_2)$ and
 $\kH_kM\simeq \bop_{i=1}^m\kF(1,\la_i)$ for some tuple $\row\la m\in{\mP^1}{(m)}$.
 Then there is an cofinite subset $\La\sbe\mP^1$ such that all direct sums
 $\bop_{i=1}^m\kF(1,\mu_i)$, where $\row\mu m\in\La^{(m)}$, are isomorphic to
 $\kT_kM'$ for some $M\in{\rpp}^\kT_0(P_1,P_2)$.

 \begin{lemma}\label{l47}
 Let $\bB$ be a tame hereditary algebra, $N$ be a partial tilting object from $\cD^b(\bB)$
 such that all direct summands of $N$ are regular. Then $\End_\bB N\simeq\prod_{i=1}^k\bB_i$,
 where each $\bB_i$ is the path algebra of a quiver of type $\rA_{n_i}$ for some $i$.
 \end{lemma}
 \begin{proof}
 Since all regular modules belong to tubes and there are no nonzero morphisms or extensions
 between modules from different tubes, we may suppose that there is a unique tube $\cT$ such
 that all direct summands of $N$ belong to its shifts. Moreover, we may suppose that
 $\Hom_\bB(N',N'')\ne0$ or $\Hom_\bB(N'',N')\ne0$ for any nontrivial decomposition
 $N=N'\+N''$. If $\cT\simeq\rep(\bH_r)$, an easy explicite calculation shows that if
 $\Hom_{\cD\bB}(N,N[k])=0$ for $k\ne0$, then $\End_\bB N$ is isomorphic to a quiver
 algebra of type $\rA_n$ with $n<r$.
 \end{proof}

  \begin{theorem}\label{t48}
  Let $\bA=\bA_1[\bW]\bA_2$, where $\bW$ is a Euclidean bimodule projective as
 $\bA_1$-module. Every space $\rpp(P_1,P_2)$ contains an open dense subset $U$
 such that every module $M\in U$ is either partial tilting or splits as $M'\+M''$, where
 $M''$ is a fixed partial tilting module, $\End_\bA M''$ is a Dynkinian algebra,
 $\Hom_\bA(M',M'')=\Hom_\bA(M'',M')=0$, and $\kT M'$ runs through
 $\bop_{i=1}^m\kF(1,\la_i)$ with $\row\la m\in\mX^{(m)}$ for some cofinite subset
 $\mX\sbe\mP^1$.
 \end{theorem}
 \begin{proof}
 Consider the open subset $\rpp_0=\rpp(P_1,P_2)\cap\rep_0(P_1,P_2)$.
 Let $M\in\rpp_0$, $M=\bop_{i=1}^nM_i$ with $M_i\in\bA\ind$. Then $\Hom_{\cD\bA}(M_i,M_j[k])=0$
 for $i\ne j$ and $k\ne0$. Let $\bB$ be the quiver algebra derived equivalent to $\bA$, $N_i$ be the
 object of the derived category $\cD^b(\bB\md)$ corresponding to $M_i$, so also 
 $\Hom_{\cD\bB}(N_i,N_j[k])=0$ for $i\ne j,\ k\ne0$. Lemma~\ref{l44} implies
 that if $N_i\in\cC[m]$ for some $i,m$, then all $N_j$ are bricks without self-extensions,
 hence the same is true for $M_j$ and $M$ is a partial tilting module.

 Suppose now that all $N_i$ are shifts of regular modules. Let $\dim N_i$ be imaginary roots for
 $1\le i\le m$ and real roots for $m<i\le n$; $M'=\bop_{i=1}^m M_i$, $M''=\bop_{i=m+1}^nM_i$.
 Lemma~\ref{l44}\,(3) implies that $\Hom_\bA(M',M'')=\Hom_\bA(M'',M')=0$ and $M''$ is a partial
 tilting (in particular, rigid) modle. Moreover, Lemma~\ref{l47} implies that $\End_\bA M''$ is a
 Dynkinian algebra. Since $N_i$ are bricks, $N_i\simeq\kF(1,\la_i)[k]\ (1\le i\le m)$ for
 some $k\in\mZ$, $\la_i\in\mP^1$ and  $\la_i\ne\la_j$ for $1\le i< j\le m$
 (it follows from \cite[Proposition~IV.7.1]{hap} that $k$ is common for all $i$). Then
 Corollary~\ref{c46} implies that there is a cofinite subset $\mX\sbe\mP^1$ and an
 open subset $U\sbe\rpp(P_1,P_2)$ such that $\kT M$ runs through
 $\kT M''\+\bop_{i=1}^m\kF(1,\la_i)$, where $\row\la m\in\mX^{(m)}$.
  \end{proof}
 
 \section{Representations of linear groups over algebras}
 \label{s5}

 Let now $\bA$ be a finite dimensional algebra over the field $\mC$ of complex numbers,
 $\sG=\gl(P,\bA)$ be a linear group over $\bA$, where $P\in\bA\pro$, $\hG$
 be the dual space of $\sG$, i.e. the space of its irreducible unitary representations and 
 $\mu_\sG$ be the \emph{Plancherel measure} on $\hG$ \cite{kir,kl}. We call a subset
 $\tG\sbe\hG$ \emph{fat} if it is open, dense and support of the Plancherel measure, i.e.
 $\mu_\sG(\hG\=\tG)=0$.

 Suppose that there is an idempotent $e\in\bA$ such that $(1-e)\bA e=0$ and set
 $\bA_1=e\bA e,\, \bA_2=(1-e)\bA(1-e),\, \bW=e\bA(1-e),\, P_1=eP,\, P_2=(1-e)P$.
 Then $\bA\simeq \bA_1[\bW]\bA_2$ and $\sG\simeq\sH\lt\sN$, where
 $\sH=\gl(P_1,\bA_1)\xx \gl(P_2,\bA_2)$, $\sN\simeq\bW(P_1,P_2)$, hence,
 \begin{align*}
  \hN=\Hom_\mC(\sN,\mC)&\simeq
	\Hom_\mC(\hP_1\*_{\bA_1}\bW\*_{\bA_2}P_2,\mC)\simeq\\
 &\simeq \Hom_{\bA_1}(\hP_1,\Hom_\mC(\bW\*_{\bA_2}P_2,\mC)) \simeq\\
 &\simeq \Hom_{\bA_1}(\hP_1,\Hom_{\bA_2}(P_2,\bW^*))
	\simeq \bW^*(\hP_1,\hP_2).
 \end{align*}
 The natural action of the group $\sH$ on the space $\hN$ coincides with its action on
 $\bW^*(\hP_1,\hP_2)$. Recall that, by the Mackey's theorem \cite{kir,kl}, there is a
 surjection $\pi:\hG\to\hN/\sH$ such that $\pi^{-1}(\sH\chi)\simeq\hH_\chi$,
 where  $\sH_\chi=\setsuch{h\in\sH}{h\chi=\chi}$. Moreover, $\pi$ is continuous
 and compatible with the Plancherel measure; especially, if $\mu_\sN(X)=0$ for a subset
 $X\sbe\hN$ and $\bar X$ is the image of $X$ in $\hN/\sH$, then
 $\mu_\sG(\pi^{-1}\bar X)=0$. Note that $\dim\sH_w<\dim\sG$ whenever
 $\sN\ne0$.

 In particular, if $\bA$ is derived equivalent to a hereditary algebra, it follows from
 \cite[Lemma~IV.1.10]{hap} that there is an idempotent $e\in\bA$ such that $e\bA e=\mC$
 and $(1-e)\bA e=0$. Then $\bA_1=\mC$, so $\bW$ is projective over $\bA_1$
 and all results from the preceding sections can be applied to this situation. It gives the
 following results. Moreover, if $\bA$ is not semisimple, $e$ can be so chosen that
 $\bW\ne 0$; otherwise $\gl(P,\bA)\simeq\prod_i\gl(d_i,\Mk)$ for every $P$.

 \begin{theorem}[cf.~\cite{dr}]\label{t51}
  Let $\bA$ be a Dynkinian algebra, $\sG=\gl(P,\bA)$ be a linear group over $\bA$.
 There is a fat subset $\tG\sbe\hG$ such that $\tG\simeq\prod_{i=1}^s\hgl(d_i,\mC)$
 for some values of $\lst ds$.
 \end{theorem}
 \begin{proof}
  We use that above notations. Consider the open sheet $Z\sbe\ld\bW(P_1,P_2)$.
 Lemma~\ref{l42} implies that it consists of a unique $\sH$-orbit and if $w\in Z$,
 then $\bA'=\End_{\ld\bW}w$ is a Dynkinian algebra. Therefore, $\sH_w$ is again
 a linear group over a Dynkinian algebra, and $\dim\sH_w<\dim\sG$ if $\bA$ is not
 semisimple. Let $\hG'=\pi^{-1}(Z)$; then $\hG'\simeq\hH_w$.
 An easy induction accomplishes the proof.
 \end{proof} 

 \begin{theorem}\label{t52}
   Let $\bA$ be a Euclidean algebra, $\sG=\gl(P,\bA)$ be a linear group over $\bA$.
 There is a fat subset $\tG\sbe\hG$ such that 
  \[
 \tG\simeq\prod_{i=1}^s\hgl(d_i,\mC) \xx \mX^{(m)}/\sS_m \xx (\mC^\xx)^m 
 \]
 for some values of $\lst ds$, $m$ and some cofinite subset $\mX\sbe\mP^1$.
 \end{theorem}
 \begin{proof}
 Let $Z^\kT_0$ be the $\kT$-sheet of $\ld]\bW(P_1,P_2)$. Theorem~\ref{t48}
 implies that either $Z^\kT_0$ consists of a unique orbit and $\End_{\ld\bW}=\bA'$
 is also a Euclidean or Dynkinian algebra, or $Z^\kT_0$ contains an open dense
 subset $U$ such that, for every $w\in U$,
 $\End_{\ld\bW}w\simeq \bA'\xx(\mC^\xx)^m$, where $\bA'$ is a Dynkinian algebra
 and $U\simeq\mX^{(m)}$ for some cofinite subswet $\mX\sbe\mP^1$. Moreover, in the
 latter case two elements $\fX,\fX'$ from $\mX^{(m)}$ belong to the same $\sH$-orbit \iff
 they only differ by the ordering of its component, i.e. belong to the same orbit of $\sS_m$.
 Again an easy induction, together with Theorem~\ref{t51} accomplishes the proof.
 \end{proof}

 \end{document}